\documentclass[11pt,a4paper]{article}

\usepackage{amssymb,amsmath,amsthm,graphicx}
\usepackage{intcalc}
\usepackage{todonotes,enumerate}
\usepackage{authblk,cite}
\usepackage[ruled,vlined,linesnumbered]{algorithm2e}
\usepackage{etoolbox}
\providetoggle{long}
\usepackage[none]{hyphenat}
\settoggle{long}{true}
\usepackage{amssymb,graphicx}
\usepackage[plain]{fullpage}
\usepackage[numbers]{natbib}
\usepackage{xcolor}
\usepackage{bbm,hyperref}
\hypersetup{
	colorlinks=true,
    linkcolor={red!50!black},
    citecolor={blue!50!black},
    urlcolor={blue!80!black},
	bookmarksopen=true,
	bookmarksnumbered,
	bookmarksopenlevel=2,
	bookmarksdepth=3
}
\usepackage{cleveref}[capitalise,2011/12/24]

\newtheorem{theorem}{Theorem}[section]
\newtheorem{corollary}[theorem]{Corollary}

\newtheorem{observation}[theorem]{Observation}

\newtheorem{fact}[theorem]{Fact}
\newtheorem{question}{Question}

\newtheorem{proposition}[theorem]{Proposition}
\newtheorem{lemma}[theorem]{Lemma}
\theoremstyle{definition}

\newtheorem{remark}[theorem]{Remark}
\newtheorem{example}[theorem]{Example}
\newcommand{\HH}{\ensuremath{\mathcal{H}}}

\def\go{\omega}

\def\gD{\Delta}

\def\cH{\mathcal{H}}

\graphicspath{{./figures/}}

\title{Conformal Hypergraphs: Duality and Implications for the Upper Clique Transversal Problem}

\author{Endre Boros\\
\small MSIS Department and RUTCOR, Rutgers University, New Jersey, USA\\
\small \texttt{endre.boros@rutgers.edu}\\
\and
Vladimir Gurvich\\
\small RUTCOR, Rutgers University, New Jersey, USA\\
\small National Research University Higher School of Economics, Moscow, Russia\\
\small \texttt{vladimir.gurvich@gmail.com}\and
Martin Milani\v c\\
\small FAMNIT and IAM, University of Primorska, Koper, Slovenia\\
\small \texttt{martin.milanic@upr.si}
\and
Yushi Uno\\
\small Graduate School of Informatics,
\small Osaka Metropolitan University,
\small Sakai, Osaka, Japan\\
\small \texttt{yushi.uno@omu.ac.jp}
}
\date{}

\begin{document}
\maketitle
\begin{abstract}
\begin{sloppypar}
Given a hypergraph $\HH$, the dual hypergraph of $\HH$ is the hypergraph of all minimal transversals of $\HH$.
The dual hypergraph is always Sperner, that is, no hyperedge contains another.
A special case of Sperner hypergraphs are the conformal Sperner hypergraphs, which correspond to the families of maximal cliques of graphs.
All these notions play an important role in many fields of mathematics and computer science, including combinatorics, algebra, database theory, etc.
Motivated by a question related to clique transversals of graphs, we study in this paper conformality of dual hypergraphs and prove several results related to the problem of recognizing this property.
We show that the problem is in {\sf co-NP} and can be solved in polynomial time for hypergraphs of bounded dimension. 
For dimension $3$, we show that the problem can be reduced to {\sc $2$-Satisfiability}.
Our approach has an application in algorithmic graph theory: we obtain a polynomial-time algorithm for recognizing graphs in which all minimal transversals of maximal cliques have size at most $k$, for any fixed~$k$.


\end{sloppypar}
\end{abstract}

 \clearpage
 \tableofcontents
 \clearpage

\section{Introduction}

A \emph{hypergraph} is a finite set of finite sets called \emph{hyperedges}.
We consider the following two properties of hypergraphs.
A hypergraph is \emph{Sperner}~\cite{MR1544925} (also called \emph{simple}~\cite{zbMATH03400923,MR1013569} or a \emph{clutter}~\cite{zbMATH01859168}) if no hyperedge is contained in another hyperedge.
A hypergraph is \emph{conformal} if for each set $U$ of vertices, if each pair of vertices in $U$ is contained in some hyperedge, then $U$ is contained in some hyperedge (see, e.g.,~\cite{zbMATH01859168}).
Both notions play an important role in combinatorics and in many other fields of mathematics and computer science.
For example, Sperner hypergraphs and their extensions have numerous applications in algebra, theory of monotone Boolean functions, and databases
(see, e.g., Anderson~\cite{MR0892525} and Engel~\cite{MR1429390}).
Furthermore, conformal hypergraphs are important for databases (see, e.g.,
Beeri, Fagin, Maier, and Yannakakis~\cite{BeeriFMY83}) and arise naturally in algebraic topology (see Berge~\cite[p.~412, Exercise~1]{zbMATH03400923}).

It is interesting to investigate the above properties in relation with the concepts of blocking and antiblocking hypergraphs.
Given a hypergraph $\HH = (V,E)$, the \emph{blocking hypergraph} (or \emph{blocker}; see, e.g., Schrijver~\cite{zbMATH01859168}) of $\HH$ is the hypergraph with vertex set $V$ whose hyperedges are exactly the minimal sets of vertices that contain at least one vertex from each hyperedge.
This natural concept was studied under several other names in the literature, including \emph{transversal hypergraph} (see Berge~\cite{zbMATH03400923,MR1013569}),  \emph{hitting sets} (see Karp~\cite{MR0378476} and also Garey and Johnson~\cite{MR0519066}), or \emph{Menger dual} (see Woodall~\cite{zbMATH03606472}).
Furthermore, motivated by the equivalent concept of monotone Boolean duality (see, e.g., Crama and Hammer~\cite{zbMATH05852793}), the blocker of $\HH$ is also called the \emph{dual hypergraph} of $\HH$ and denoted by $\HH^d$.
Indeed, in the case of Sperner hypergraphs, the operation of mapping $\HH$ to its dual hypergraph is an involution, that is, $(\HH^d)^d = \HH$ (see, e.g., Berge~\cite{zbMATH03400923} and Schrijver~\cite{zbMATH01859168}).
Hypergraph duality has many applications, for example to Nash-solvability of two-person game forms; see Edmonds and Fulkerson~\cite{zbMATH03346366} for the zero-sum case, and Gurvich and Naumova~\cite{gurvich2022lexicographically} for the general two-person case.
Many other applications and references can be found in the papers by Eiter and Gottlob~\cite{MR1361157} and Makino and Kameda~\cite{zbMATH01670337}.
The complexity of the \emph{dualization problem}, that is, computing the dual hypergraph $\HH^d$ given $\HH$, is a notorious open problem (see Fredman and Khachiyan~\cite{MR1417667}).

Similarly to the blocker of a given hypergraph $\HH = (V,E)$, one can define the \emph{antiblocker} of $\HH$ as the hypergraph $\HH^a$ with vertex set $V$ whose hyperedges are exactly the maximal sets of vertices that contain at most one vertex from each hyperedge (see Fulkerson~\cite{zbMATH03360187}).
The antiblocker was also called
\emph{K\"{o}nig dual} by
Woodall~\cite{zbMATH03606472}; see also McKee~\cite{zbMATH03811629}.
Blockers and antiblockers are related to perfect graphs and polyhedral combinatorics and were considered together in several papers~\cite{zbMATH03402379,MR2755907,zbMATH03637903,zbMATH02016912,zbMATH06256806}.

It follows easily from the definitions that for every hypergraph $\HH$, its dual $\HH^d$ is always Sperner.
Furthermore, as explained above, if $\HH$ is also Sperner, then $(\HH^d)^d = \HH$.
Analogously, for every hypergraph $\HH$, its antiblocker $\HH^a$ is always conformal, and if $\HH$ is also conformal, then $(\HH^a)^a = \HH$, as shown by Woodall~\cite{zbMATH03606472,zbMATH03674127} (see also Schrijver~\cite{zbMATH01859168}).
However, while the antiblocker $\HH^a$ is always Sperner, the dual $\HH^d$ need not be conformal.
For example, all the $2$-element subsets of a $3$-element set form a hypergraph such that its dual is not conformal.
Moreover, even if a hypergraph is conformal, its dual may fail to be conformal.\footnote{Consider the $2$-uniform hypergraph $\HH$ given by the edges of the $5$-cycle, that is, $\HH = (V,E)$ with $V = \{1,2,3,4,5\}$ and $E = \{\{1,2\},\{2,3\},\{3,4\},\{4,5\},\{5,1\}\}$.
Clearly, $\HH$ is conformal.
However, $E(\HH^d) = \{
\{1,2,4\},
\{2,3,5\},
\{3,4,1\},
\{4,5,2\},
\{5,1,3\}
\}$.
In particular, every pair of vertices belongs to a hyperedge and hence $\HH^d$ is not conformal.}

\begin{table}[h!]
    \centering
    \caption{Properties of blockers and antiblockers.}
    \label{tab:my_label}
    \medskip
    \begin{tabular}{c|c|c}
         & Sperner & conformal \\
         \hline
        blocker, $\HH^d$ & always & not always \\
        \hline
        antiblocker, $\HH^a$ & always & always \\
    \end{tabular}
\end{table}

\begin{sloppypar}
The above relations, summarized in \Cref{tab:my_label}, motivate the study of hypergraphs whose dual is conformal.
Variants of dual conformality are important for the dualization problem (see Khachiyan, Boros, Elbassioni, and Gurvich~\cite{DBLP:conf/cocoon/KhachiyanBEG05,MR2352109,MR2287281}).
Conformal hypergraphs were characterized independently by Gilmore~\cite{gilmore1962families} (see also~\cite{zbMATH03400923,MR1013569}) and Zykov~\cite{zbMATH03485854}; the characterization leads to a polynomial-time recognition algorithm.
On other other hand, the complexity of recognizing hypergraphs whose dual is conformal is open.
In this paper we focus on this problem and call such hypergraphs \emph{dually conformal}.

We prove several results related to the problem of recognizing dually conformal hypergraphs.
After observing that the problem belongs to {\sf co-NP}, we develop our first main result, a polynomial-time algorithm for the case of hypergraphs of bounded dimension (maximum size of a hyperedge).
For hypergraphs of dimension at most~$3$ we develop an alternative approach based on {\sc $2$-Satisfiability}.
We also discuss separately the case of $2$-uniform hypergraphs, that is, the case of graphs.
\end{sloppypar}

Our second main result is related to a question on clique transversals in graphs, which is in fact our main motivation for the study of dually conformal hypergraphs.
More precisely, using another polynomially solvable case of the recognition problem of dually conformal hypergraphs, we obtain a polynomial-time algorithm for the following problem, for any fixed positive integer~$k$: Given a graph $G$, does $G$ admit a minimal clique transversal (that is, an inclusion-minimal set of vertices that intersects all maximal cliques) of size at least $k$?
This problem was studied recently by Milani\v{c} and Uno~\cite{MilanicUnoWG2023} and was shown to be {\sf NP}-hard when $k$ is part of the input.

\subsection*{Structure of the paper}

\begin{sloppypar}
In \Cref{sec:prelim} we summarize the necessary preliminaries, including some basic properties of conformal hypergraphs, both in the Sperner case and in general.
In \Cref{sec:dually-conformal} we present some basic results about dually conformal hypergraphs and initiate a study of the corresponding recognition problem by showing that the problem belongs to {\sf co-NP} and identifying a  polynomially solvable special case.
Applications of this algorithm to graphs are presented in \Cref{sec:small-upper-clique-transversal}.
In \Cref{sec:bounded-dimension} we develop a polynomial-time algorithm for the recognition of dually conformal hypergraphs of bounded dimension, with alternative approaches for the case of dimension $3$ and the $2$-uniform case.
We conclude the paper in \Cref{sec:discussion} with a discussion and several open questions.
\end{sloppypar}

\section{Preliminaries}\label{sec:prelim}

\subsection{Notation and definitions}

A \emph{hypergraph} is a pair $\HH = (V,E)$ where $V$ is a finite set of \emph{vertices} and $E$ is a set of subsets of $V$ called \emph{hyperedges} such that every vertex belongs to a hyperedge.
For a hypergraph $\HH = (V,E)$ we write $E(\HH) = E$ and $V(\HH) = V$, and denote by $\dim(\cH)=\max_{e\in E} |e|$ its \emph{dimension}.
A hypergraph $\HH$ is said to be \emph{$k$-uniform} if $|e| = k$ for all $e\in E(\HH)$.
Thus, $2$-uniform hypergraphs are precisely the (finite, simple, and undirected) graphs without isolated vertices.
We only consider graphs and hypergraphs with nonempty vertex sets.
For a vertex $v\in V$ its degree $\deg(v)=\deg_\cH(v)$ is the number of hyperedges in $E$ that contain $v$ and $\gD(\cH)=\max_{v\in V} \deg(v)$ is the maximum degree of $\cH$.
The \emph{size} of a hypergraph $\HH$ is the number of hyperedges in $\HH$.
A hyperedge of $\HH$ is said to be \emph{maximal} if it is not contained in any other hyperedge.
A hypergraph is \emph{Sperner} if no hyperedge contains another, or, equivalently, if every hyperedge is maximal.

A \emph{transversal} of a hypergraph $\HH = (V,E)$ is a set of vertices intersecting all hyperedges. A transversal is \emph{minimal} if it does not contain any other transversal.
Recall that the \emph{dual hypergraph} of a hypergraph $\HH = (V,E)$ is the hypergraph $\HH^d$ with vertex set $V$, whose hyperedges are exactly the minimal transversals of $\HH$.

\begin{fact}[Folklore; see, e.g., Berge~\cite{MR1013569}]\label{duality-is-an-involution}
Let $\HH$ be a Sperner hypergraph. Then $(\HH^d)^d = \HH$.
\end{fact}

\subsection{Representation of hypergraphs}\label{sec:representation}

In this subsection we describe a useful data structure for representing hypergraphs.
Let $\cH=(V,E)$ be a hypergraph.
We write $n=|V|$ and $m=|E|$.
An \emph{incident pair} of $\cH$ is a pair $(v,e)$ such that $v\in e\in E$.
We assume that $\cH$ is represented by a complete list of its edges, as subsets of $V$, and equipped with a fixed pair of orderings of its vertices and edges, say $V = \{v_1,\ldots,  v_n\}$ and $E = \{e_1,\ldots, e_m\}$.

We first perform a preprocessing step taking time $\mathcal{O}(|V||E|)$ in order to compute the {\it edge-vertex incidence matrix} of $\cH$, a binary matrix $I\in \{0,1\}^{E\times V}$ with rows indexed by the hyperedges of $\cH$, columns indexed by the vertices of $\cH$, and $I_{e,v} = 1$ if and only if $v\in e$.
Having constructed the edge-vertex incidence matrix, we can look up in constant time whether, given a vertex $v\in V$ and hyperedge $e\in E$, the pair $(v,e)$ is an incident pair of $\cH$.

Next we construct a {\it doubly linked representation of incident pairs} of $\cH$, that is, a collection $L$ of doubly linked lists of incident pairs, one for each vertex and one for each hyperedge.
Each incident pair contains a pointer to its vertex, another one to its hyperedge, and has four links -- horizontal \texttt{prev} and \texttt{next} and vertical \texttt{prev} and \texttt{next}.
The horizontal links form a doubly linked circular list attached to the hyperedge, and the vertical ones form a doubly linked circular list attached to the vertex. See \Cref{fig:example-hypergraph} for an example.
Due to the doubly linked nature, insertions can be done in constant time.
We can thus build the structure $L$ in $\mathcal{O}(|V||E|)$ time, as follows.
\begin{enumerate}
\item First, we initialize the doubly linked lists for each vertex and hyperedge to be the doubly linked lists consisting only of the corresponding vertex, resp.~hyperedge.
\item Then, we traverse the edge-vertex incidence matrix $I$ row by row.
As we traverse a row labeled by a hyperedge $e$, we build the doubly linked list corresponding to this hyperedge (with horizontal \texttt{prev} and \texttt{next} links) along with the pointers to $e$.
At the same time, when a new incident pair $(v,e)$ is added to the list, the doubly linked list corresponding to the vertex $v$ is augmented with this pair (with vertical \texttt{prev} and \texttt{next} links) and the pointer to vertex $v$.
\end{enumerate}

\begin{figure}[h!]
	\centering
	\includegraphics[width=0.85\textwidth]{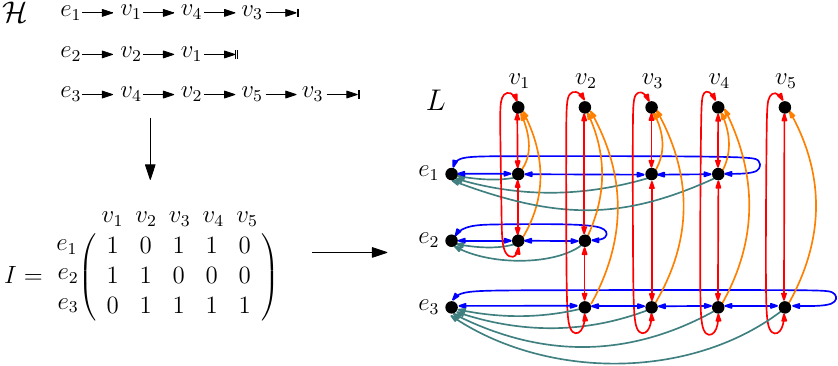}
	\caption{A hypergraph $\cH$, its edge-vertex incidence matrix, and the doubly linked representation of its incident pairs.}
	\label{fig:example-hypergraph}
\end{figure}

The usefulness of the above data structures is summarized in the following.

\begin{sloppypar}
\begin{proposition}\label{data-structure-for-hypergraphs}
Given a hypergraph $\cH = (V,E)$, with $V = \{v_1,\ldots,  v_n\}$ and $E = \{e_1,\ldots, e_m\}$, there is an algorithm running in time $\mathcal{O}(|V||E|)$ that computes its edge-vertex incidence matrix and the doubly linked
representation of its incident pairs.

Using the incidence matrix we can test in constant time the relation $v\in e$ for all $v\in V$ and $e\in E$.
Using the doubly linked representation of the incident pairs we can:
\begin{itemize}
\item list the vertices of a hyperedge $e\in E$ in time linear in $|e|\leq \dim(\cH)$;
\item list the hyperedges containing a vertex $v\in V$ in time linear in $\deg(v)\leq \gD(\cH)$;
\item compute for any two hyperedges $e$ and $f$ their union, intersection, and the two set differences $e\setminus f$ and $f\setminus e$ in time $\mathcal{O}(|e|+|f|) = \mathcal{O}(\dim(\cH))$; in particular, we can test in time $\mathcal{O}(\dim(\cH))$ if $e\subseteq f$.
\end{itemize}
\end{proposition}
\end{sloppypar}

Let us also remark that, when discussing the running times of algorithms on graphs (in \Cref{sec:small-upper-clique-transversal}), we assume that the adjacency lists are sorted.
If they are initially not sorted, we first sort them in time $\mathcal{O}(|V|+|E|)$ (see~\cite{MR2063679}).

\subsection{Subtransverals}

Given a hypergraph $\HH = (V,E)$, a set $S\subseteq V$ is a \emph{subtransversal} of $\HH$ if $S$ is a subset of a minimal transversal.
The following characterization of subtransversals due to Boros, Gurvich, and Hammer~\cite[Theorem 1]{MR1754735} was formulated first in terms of prime implicants of monotone Boolean functions and their duals, and re-proved in terms of hypergraphs in~\cite{MR1841121}.
Given a set $S\subseteq V$ and a vertex $v\in S$, we denote by $E_v(S)$ the set of hyperedges $e\in E$ such that $e\cap S = \{v\}$.

\begin{theorem}[Boros, Gurvich, Elbassioni, and Khachiyan~\cite{MR1841121}]\label{subtransversal-characterization}
Let $\mathcal{H} = (V,E)$ be a hypergraph and let $S\subseteq V$.
Then $S$ is a subtransversal of $\HH$ if and only if there exists a collection of hyperedges
$\{e_v\in E_v(S) : v\in S\}$ such that the set $(\bigcup_{v\in S}e_v)\setminus S$
does not contain any hyperedge of $\HH$.
\end{theorem}

\begin{sloppypar}
Note that edges that intersect $S$ in more than one vertex do not influence the fact whether $S$ is a subtransversal or not.
The problem of determining if a given set $S$ is a subtransversal is {\sf NP}-complete even for $2$-uniform hypergraphs (see~\cite{MR1754735,MR1841121}).
For sets of bounded cardinality, however,
\Cref{subtransversal-characterization} leads to a polynomial-time algorithm.
\end{sloppypar}

\begin{sloppypar}
\begin{corollary}\label{subtransversal-running-time}
Let $\mathcal{H} = (V,E)$ be a hypergraph with dimension $k$ and maximum degree~$\gD$, given by an edge-vertex incidence matrix and a doubly linked representation of its incident pairs, and let $S\subseteq V$.
Then, there exists an algorithm running in time
\[
\mathcal{O}\left(k|E|\cdot\min\left\{\gD^{|S|}\,,\left(\frac{|E|}{|S|}\right)^{|S|}\right\}\right)
\]
that determines if $S$ is a subtransversal of $\HH$.
In particular, if $|S| = \mathcal{O}(1)$, the complexity is $\mathcal{O}(k|E|\Delta^{|S|})$.
\end{corollary}
\end{sloppypar}

\begin{proof}
Note that any minimal transversal has at most as many vertices as the number of hyperedges.
Thus, if $|S|>|E|$, then we can determine in time $\mathcal{O}(|S|+|E|)= \mathcal{O}(k|E|)$ that $S$ is not a subtransversal.
Note that $\mathcal{O}(|S|) = \mathcal{O}(|V|) =\mathcal{O}(k|E|)$, since $V = \bigcup_{e\in E}e$.

From now on, we assume that $|S|\leq |E|$.
To a subset $S\subseteq V$ we associate the following families of edges:
\[
\begin{array}{rll}
E_v(S) &=~ \{e\in E\mid e\cap S=\{v\}\} &\text{ for }~~~ v\in S, \text{ and, }\\
E_\go(S) &=~ \{e\in E\mid e\cap S=\emptyset\}.
\end{array}
\]
We describe the desired algorithm with the following procedure.
For convenience, we also include, in smaller font, a time complexity analysis of each step.

\bigskip
\noindent\textsc{Procedure SubTransversal:}
\begin{description}
	\item[\hspace*{5mm}Input:] A hypergraph $\cH=(V,E)$ given by an edge-vertex incidence matrix and a doubly linked representation $L$ of its incident pairs, a subset $S\subseteq V$ such that $|S|\leq |E|$.
	\item[\hspace*{5mm}Output:] \textsc{Yes} if $S$ is a subset of a minimal transversal of $\cH$, and \textsc{No} otherwise.

	\item[\hspace*{5mm}Step 1:] Compute the families $E_u(S)$ for $u\in S\cup\{\go\}$.
  \begin{itemize}
      \item[] {\small \it This can be done in time $\mathcal{O}(|S|+k|E|)=\mathcal{O}(k|E|)$, as follows. We first traverse the set $S$ and mark each vertex that belongs to $S$.
		 Then for each hyperedge $e\in E$ we traverse the corresponding list of $\mathcal{O}(k)$ vertices; if the hyperedge $e$ contains no vertex from $S$, we put it in $E_{\omega}(S)$, and if it contains a unique vertex from $S$, say $v$, we put it in $E_v(S)$.}
  \end{itemize}
	\begin{description}
		\item[1.1] If $E_v(S)=\emptyset$ for some $v\in S$, then \textsc{STOP} and output \textsc{No}.
  \begin{itemize}
      \item[]  {\small \it This can be done in time $\mathcal{O}(|S|)$.}
 \end{itemize}
		\item[1.2] Otherwise, if $E_\go(S)=\emptyset$, then \textsc{STOP} and output \textsc{Yes} ($S$ is a minimal transversal of $\cH$ in this case).
    \begin{itemize}
      \item[]   {\small \it This can be done in time $\mathcal{O}(1)$.}
 \end{itemize}
	\end{description}
	\item[\hspace*{5mm}Step 2:]
	Initialize an array $A\in \{0,1\}^V$ of length $n$ by zeros.
      \begin{itemize}
      \item[]   {\small \it This can be done in time $\mathcal{O}(|V|) = \mathcal{O}(k|E|)$.
	(Recall that $|V|\le k|E|$, since $V = \bigcup_{e\in E}e$.)}
 \end{itemize}
	For each selection $e_v\in E_v(S)$, $v\in S$:
      \begin{itemize}
      \item[]  {\small \it The number of such selections is $\displaystyle\prod_{v\in S}|E_v(S)| \leq \min \left\{\gD^{|S|}, \left(\frac{|E|}{|S|}\right)^{|S|}\right\}$.}
 \end{itemize}

	\begin{description}
		\item[2.1] Compute $\displaystyle U=\bigcup_{v\in S} e_v$.

     \begin{itemize}
      \item[]  {\small \textit{This can be done in time $\mathcal{O}(k|S|)$, as follows.
       We first create an object for $U$ with a root of a doubly linked list that is initially empty (}\texttt{prev} \textit{and} \texttt{next} \textit{point back to itself).
		We then iterate over all $v\in S$ and look up the vertices of the edge $e_v$, one by one, and for each such vertex $u\in e_v$ first check the value of $A_u$.
		If $A_u = 0$, we set $A_u = 1$ and then we add $u$, with the corresponding} \texttt{prev} \textit{and} \texttt{next} \textit{links, to the list of~$U$.
        This takes time $\mathcal{O}(k)$ per edge $e_v$ (for $v\in S$), resulting in a total time of $\mathcal{O}(k|S|)$.}
		
		\textit{Note that this procedure is correct, since at the end of the procedure, the array $A$ will have $A_u = 1$ if and only if $u\in U$.}}
 \end{itemize}
		
		\medskip
		\item[2.2] \textsc{STOP} and output \textsc{Yes} if $e\not\subseteq U$ for all $e\in E_\go(S)$.
   \begin{itemize}
      \item[]  {\small \it This can be done in time $\mathcal{O}(k|E_\go(S)|)=\mathcal{O}(k|E|)$, as follows.
      For a given $e\in E_\omega(S)$ the test $e\not\subseteq U$ can be performed in time $\mathcal{O}(|e|) = \mathcal{O}(k)$ by scanning the doubly linked list of $e$ and checking the corresponding entries of the array~$A$.}
 \end{itemize}
		
		\medskip
		
        \item[2.3] Restore the array $A$ to the all-zero array.

   \begin{itemize}
      \item[]  {\small \it This can be done in time $\mathcal{O}(k|S|)$, by
      scanning the set $U$ once in linear time in the length of this set, which is $\mathcal{O}(k|S|)$, and switching back the corresponding entries in the array $A$ to zero.}
 \end{itemize}

	\end{description}
	\item[\hspace*{5mm}Step 3:] \textsc{STOP} and output \textsc{No}.
   \begin{itemize}
      \item[]   {\small \it This can be done in time $\mathcal{O}(1)$.}
 \end{itemize}
\end{description}

Thus, we get two upper estimates for the running time of \textsc{SubTransversal}:
\[
\mathcal{O}\left(k|E|\gD^{|S|}\right) ~~~\text{ and }~~~ \mathcal{O}\left(k|E|\left(\frac{|E|}{|S|}\right)^{|S|}\right)\,,
\]
as claimed.
\end{proof}

\subsection{Conformal hypergraphs}\label{sec:conformal-hypergraphs}

\begin{sloppypar}
In this section we summarize some basic properties of conformal hypergraphs: a characterization of conformal Sperner hypergraphs, which establishes a close connection with graphs, a characterization of general conformal hypergraphs, and a polynomial-time recognition algorithm of conformal hypergraphs.
\end{sloppypar}

\begin{sloppypar}
All the graphs in this paper are finite, simple, and undirected.
We use standard graph theory terminology, following West~\cite{MR1367739}.

Given a hypergraph $\HH = (V,E)$, its \emph{co-occurrence graph} is the graph $G(\HH)$ with vertex set $V$ that has an edge between two distinct vertices $u$ and $v$ if there is a hyperedge $e$ of $\HH$ that contains both $u$ and $v$.

\begin{observation}\label{hyperedges are cliques in co-occurrence graph}
For every hypergraph $\HH$, every hyperedge of $\HH$ is a clique in the co-occurrence graph $G(\HH)$.
\end{observation}

Note, however, that hyperedges of $\HH$ are not necessarily maximal cliques of $G(\HH)$. For example, if $\HH$ is the complete graph $K_3$, then $G(\HH) = \HH$, but $G(\HH)$ has a unique maximal clique of size $3$.

Recall that a hypergraph is said to be \emph{conformal} if for each set $U$ of vertices, if each pair of vertices in $U$ is contained in some hyperedge, then $U$ is contained in some hyperedge.
It is not difficult to see that a hypergraph $\HH$ is conformal if and only if every maximal clique of its co-occurrence graph is a hyperedge of $\HH$ (in fact, this was the definition of conformality given by Berge~\cite{zbMATH03400923,MR1013569}).
Furthermore, a Sperner hypergraph $\HH$ is conformal if and only if every maximal clique of its co-occurrence graph is a hyperedge of $\HH$ (see~\cite{BeeriFMY83}).
\end{sloppypar}

We now recall a characterization of Sperner conformal hypergraphs due to Beeri, Fagin, Maier, and Yannakakis~\cite{BeeriFMY83} (see also Berge~\cite{zbMATH03400923,MR1013569} for the equivalence between properties~\ref{SC-item1} and~\ref{SC-item2}).
The \emph{clique hypergraph} of a graph $G = (V,E)$ is the hypergraph with vertex set $V$ with hyperedges exactly the maximal cliques in $G$.

\begin{theorem}[\cite{BeeriFMY83}; see also~\cite{zbMATH03400923,MR1013569}]\label{Sperner-conformal}
For every Sperner hypergraph $\HH$, the following properties are equivalent.
\begin{enumerate}
\item\label{SC-item1} $\HH$ is conformal.
\item\label{SC-item2} $\HH$ is the clique hypergraph of some graph.
\item\label{SC-item3} $\HH$ is the clique hypergraph of its co-occurrence graph.
\end{enumerate}
\end{theorem}

We now generalize \Cref{Sperner-conformal} by characterizing the conformality property for general (not necessarily Sperner) hypergraphs.

\begin{lemma}\label{a-technical-lemma}
Let $\HH$ be a hypergraph such that there exists a graph $G = (V,E)$ and a collection $\mathcal{C}$ of cliques of $G$ containing all maximal cliques of  $G$ (and possibly some others) such that $\HH = (V,\mathcal{C})$.
Then $G = G(\HH)$.
\end{lemma}

\begin{proof}
We have $V(G(\HH)) = V(\HH) = V = V(G)$.
Furthermore, two distinct vertices $u$ and $v$ are adjacent in $G$ if and only if there exists a maximal clique in $G$ containing both $u$ and $v$, and they are adjacent in the co-occurrence graph $G(\HH)$ if and only if there exists a hyperedge of $\HH$ containing $u$ and $v$.
The assumption on $\mathcal{C}$ implies that there exists a maximal clique in $G$ containing both vertices if and only if there exists a set in $\mathcal{C}$ containing both. Thus, since $E(\HH) = \mathcal{C}$, we infer that graphs $G$ and $G(\HH)$ have the same edge sets.
We conclude that $G = G(\HH)$.
\end{proof}

\begin{theorem}\label{characterization of conformality}
For every hypergraph $\HH$, the following properties are equivalent.
\begin{enumerate}
\item\label{C-item1} $\HH$ is conformal.
\item\label{C-item3} Every maximal clique in $G(\HH)$ is a maximal hyperedge of $\HH$.
\item\label{C-item4} There exists a graph $G = (V,E)$ and a collection $\mathcal{C}$ of cliques of $G$ containing all maximal cliques of  $G$ (and possibly some others) such that $\HH = (V,\mathcal{C})$.
\end{enumerate}
\end{theorem}

\begin{proof}
We show first that property~\ref{C-item1} implies property~\ref{C-item3}.
Suppose first that $\HH$ is conformal, that is, every maximal clique in $G(\HH)$ is a hyperedge of $\HH$.
Let $C$ be a maximal clique in $G(\HH)$.
Since $\HH$ is conformal, $C$ is a hyperedge of $\HH$.
It is in fact a maximal hyperedge, since if $C$ is properly contained in another hyperedge $e$ of $\HH$, then by \Cref{hyperedges are cliques in co-occurrence graph} we obtain that $e$ is a clique in $G(\HH)$ properly containing $C$, contrary to the assumption that $C$ is a maximal clique.
Thus, property~\ref{C-item3} holds.

Next, we show that property~\ref{C-item3} implies property~\ref{C-item4}. To this end, suppose that every maximal clique in $G(\HH)$ is a maximal hyperedge of $\HH$, and let $G = G(\HH)$ and $\mathcal{C} = E(\HH)$.
We then have $V(\HH) = V(G)$, by \Cref{hyperedges are cliques in co-occurrence graph} every member of $\mathcal{C}$ is a clique of $G$, and by property~\ref{C-item3}, every maximal clique in $G$ belongs to $\mathcal{C}$.
Thus, property~\ref{C-item4} holds for $G = G(\HH)$ and $\mathcal{C} = E(\HH)$.

We show next that property~\ref{C-item4} implies property~\ref{C-item1}. Suppose that there exists a graph $G = (V,E)$ and a collection $\mathcal{C}$ of cliques of $G$ containing all maximal cliques of $G$ (and possibly some others) such that $\HH = (V,\mathcal{C})$.
By \Cref{a-technical-lemma}, we have $G = G(\HH)$.
This implies that every maximal clique in $G(\HH) = G$ is a hyperedge of $\HH$, thus $\HH$ is conformal and property~\ref{C-item1} holds.
\end{proof}

Note that the proof of \Cref{characterization of conformality} shows that if $\HH = (V,\mathcal{C})$ for some graph $G = (V,E)$ and a collection $\mathcal{C}$ of cliques of $G$ containing all maximal cliques of $G$, then not only the collection $\mathcal{C}$ but also the graph $G$ is uniquely determined from $\HH$; namely, $G$ is the co-occurrence graph of $\HH$.
Checking conformality of a given hypergraph can be done in polynomial time, due to the following characterization.

\begin{theorem} [Gilmore~\cite{gilmore1962families}; see also~\cite{zbMATH03485854,zbMATH03400923,MR1013569}]
\label{conformal hypergraphs}
A hypergraph $\HH = (V,E)$ is conformal if and only if for every three hyperedges $e_1,e_2,e_3\in E$ there exists a hyperedge $e\in E$ such that $$(e_1\cap e_2)\cup (e_1\cap e_3)\cup (e_2\cap e_3)\subseteq e\,.$$
\end{theorem}

\begin{proposition}
Given a hypergraph $\mathcal{H} = (V,E)$ with dimension $k$, it can be tested in time $\mathcal{O}(|V||E|+k|E|^4)$ if $\HH$ is conformal.
\end{proposition}

\begin{proof}
Using \Cref{data-structure-for-hypergraphs}, we compute in time $\mathcal{O}(|V||E|)$ the edge-vertex incidence matrix of $\cH$ and the doubly linked representation of its incident pairs.
We then check the conformality of $\HH$ by verifying the condition from~\Cref{conformal hypergraphs}.
This can be done by iterating over all $\mathcal{O}(|E|^3)$ triples $\{e_1,e_2,e_3\}$ of hyperedges, and for each such triple computing in time $\mathcal{O}(k)$ the set $S = (e_1\cap e_2)\cup (e_1\cap e_3)\cup (e_2\cap e_3)$, and iterating over all edges $e\in E$ to verify the inclusion $S\subseteq e$.
The overall running time of this procedure is $\mathcal{O}(|E|^3\cdot(k+|E|\cdot k)) =\mathcal{O}(k|E|^4)$.
\end{proof}

\section{Dually conformal hypergraphs}\label{sec:dually-conformal}

\begin{sloppypar}
We say that a hypergraph $\HH$ is \emph{dually conformal} if its dual hypergraph $\HH^d$ is conformal.
In this section we present some basic observations about dually conformal hypergraphs and initiate a study of the corresponding recognition problem.
While we do not settle the computational complexity status of the problem, we show that the problem is in {\sf co-NP} and develop a polynomial-time algorithm for a special case.
\end{sloppypar}

\subsection{Basic observations}

Since the dual hypergraph of any hypergraph $\HH$ is the same as the dual hypergraph of the hypergraph obtained from $\HH$ by keeping only the inclusion-minimal hyperedges, in order to test dual conformality of a hypergraph we can assume without loss of generality that the hypergraph is Sperner.

In the next proposition, we characterize the dually conformal Sperner hypergraphs using a connection with graphs.
Given a graph $G$, a set of vertices that intersects all maximal cliques of $G$ is called a \emph{clique transversal} in $G$.
A clique transversal in $G$ is \emph{minimal} if it does not contain any other clique transversal.

\begin{proposition}\label{lem:clique-transversals-characterization}
Let $\HH$ be a hypergraph.
Then the following statements are equivalent.
\begin{enumerate}
\item $\HH$ is a dually conformal Sperner hypergraph.
\item There exists a graph $G$ such that $\HH$ is the hypergraph of all minimal clique transversals of~$G$.
\end{enumerate}
\end{proposition}

\begin{proof}
Let $\HH$ be a dually conformal Sperner hypergraph.
Let $G$ be the co-occurrence graph of $\HH^d$.
Since $\HH^d$ is a conformal Sperner hypergraph, \Cref{Sperner-conformal} implies that $\HH^d$ is the clique hypergraph of $G$.
But then $\HH = (\HH^d)^d$ is exactly the hypergraph of all minimal clique transversals of~$G$.

Conversely, let $G$ be a graph and let $\HH$ be the hypergraph of all minimal clique transversals of $G$.
By construction, $\HH$ is a Sperner hypergraph.
Then $\HH^d$ is the clique hypergraph of $G$ and thus $\HH^d$ is conformal.
\end{proof}

The following characterization of dually conformal hypergraphs follows immediately from the definition.

\begin{observation}\label{prop:yes-instances-of-TC}
For every hypergraph $\HH$, the following properties are equivalent.
\begin{enumerate}
\item $\HH$ is dually conformal.
\item Every maximal clique in $G(\HH^d)$ is a minimal transversal of $\HH$.
\end{enumerate}
\end{observation}

Fix a hypergraph $\HH$ and let $G = G(\HH^d)$.
By \Cref{prop:yes-instances-of-TC}, a necessary and sufficient condition for $\HH$ to be dually conformal is that every maximal clique of $G$ is a minimal transversal of $\HH$.
Thus, in general, there are two possible reasons why $\HH$ could fail to be dually conformal.

\begin{corollary}\label{no instances characterization 1}
Let $\HH$ be a hypergraph and let $G = G(\HH^d)$.
Then $\HH$ is not dually conformal if and only if one of the following two conditions holds.
\begin{enumerate}[(a)]
\item $G$ contains a maximal clique $C$ that is not a transversal of $\HH$, or
\item $G$ contains a maximal clique $C$ that is a transversal of  $\HH$ but not a minimal one.
\end{enumerate}
\end{corollary}

As shown by the following two examples, the two conditions are independent of each other.

\begin{example}\label{example-a-not-b}
The following hypergraph satisfies condition (a) but not condition (b).
Let $\HH$ be the hypergraph with vertex set $\{1,\ldots, 6\} $ and hyperedges $\{1,2\}$, $\{1,3\}$, $\{2,3\}$, $\{1,4\}$, $\{2,5\}$, $\{3,6\}$, and $\{4,5,6\}$.
Then the hyperedges of $\HH^d$ are $\{1,2,6\}$, $\{1,3,5\}$, and $\{2,3,4\}$. Its co-occurrence graph $G = G(\HH^d)$ is shown in Fig.~\ref{fig:example-1}.

\begin{figure}[h!]
	\centering
		\includegraphics[width=0.75\textwidth]{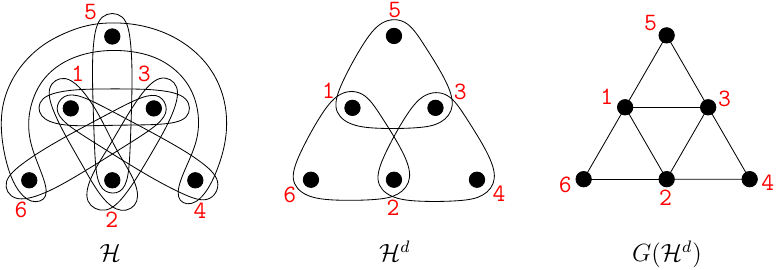}
	\caption{A hypergraph $\HH$, its dual hypergraph $\HH^d$, and the co-occurrence graph of $\HH^d$.}
	\label{fig:example-1}
\end{figure}

Note that $C = \{1,2,3\}$ is a maximal clique in $G$ that is not a transversal of $\HH$, since it misses the hyperedge $\{4,5,6\}$.
Thus, $\HH$ satisfies condition (a).
On the other hand, all maximal cliques in $G$ other than $C$ are minimal transversals of $\HH$, and hence $\HH$ does not satisfy condition (b).
\end{example}

\begin{sloppypar}
\begin{example}
The following hypergraph satisfies condition (b) but not condition (a).
Let $G$ be the complete graph $K_3$ and let $\HH = G$,
that is, $V(\HH) = \{1,2,3\}$ and $E(\HH) = \{\{1,2\},\{1,3\},\{2,3\}\}$.
Then $\HH^d = \HH$ and $G(\HH^d)= G$.
Graph $G$ is complete and hence contains a unique maximal clique $C$, namely $C = \{1,2,3\}$.
This clique is a transversal of $\HH$ but not a minimal one. Thus, $\HH$ satisfies condition (b) but not condition (a).
\end{example}
\end{sloppypar}

Furthermore, as shown by the following example, the two conditions can occur simultaneously.

\begin{sloppypar}
\begin{example}
The following hypergraph satisfies both conditions (a) and (b).
Let $\HH$ have vertex set $\{1,\ldots, 6\} $ and hyperedges $\{1,4,5\}$, $\{1,4,6\}$, $\{2,4,5\}$, $\{2,5,6\}$, $\{3,4,6\}$, $\{3,5,6\}$, and $\{4,5,6\}$.
Then the hyperedges of $\HH^d$ are $\{1,2,6\}$, $\{1,3,5\}$, $\{2,3,4\}$, $\{4,5\}$, $\{4,6\}$, and $\{5,6\}$.
Its co-occurrence graph $G = G(\HH^d)$ is isomorphic to the complete multipartite graph $K_{2,2,2}$, with parts $\{1,4\}$, $\{2,5\}$, and $\{3,6\}$; two vertices in $G$ are adjacent to each other if and only if they belong to different parts.

Note that the set $C = \{1,2,3\}$ is a maximal clique in $G$ that is not a transversal of $\HH$, since it misses the hyperedge $\{4,5,6\}$.
Thus, $\HH$ satisfies condition (a).
Furthermore, $C' = \{4,5,6\}$ is a maximal clique in $G$ that is a transversal of $\HH$ but not a minimal one, since it properly contains the minimal transversal $\{4,5\}$. Hence, $\HH$ also satisfies condition~(b).
\end{example}
\end{sloppypar}

\subsection{Computing the co-occurrence graph of the dual hypergraph}

Immediately from \Cref{hyperedges are cliques in co-occurrence graph} we obtain the following.

\begin{corollary}\label{hyperedges are cliques in co-occurrence graph-corollary}
Every hyperedge of $\HH^d$ is a clique in $G(\HH^d)$.
\end{corollary}

\begin{proposition}\label{dual-co-occurrence-graph-in-poly-time-improved}
Given a hypergraph $\mathcal{H} = (V,E)$ with dimension $k$ and maximum degree~$\gD$, the co-occurrence graph of the dual hypergraph $\HH^d$ can be computed in time $\mathcal{O}(k|E|\Delta^2|V|^{2})$.
\end{proposition}

\begin{proof}
Using \Cref{data-structure-for-hypergraphs}, we compute in time $\mathcal{O}(|V||E|)$ the edge-vertex incidence matrix of $\cH$ and the doubly linked representation of its incident pairs.
Two distinct vertices $u$ and $v$ in $V$ are adjacent in the co-occurrence graph of $(\HH)^d$ if and only if the set $\{u,v\}$ is a subtransversal of $\HH$.
Applying \Cref{subtransversal-running-time} we can test in time $\mathcal{O}(k|E|\Delta^{2})$ if any such set is a subtransversal of $\HH$.
As the total number of pairs is $\mathcal{O}(|V|^2)$, the claimed time complexity follows.
\end{proof}

\begin{corollary}\label{dual-co-occurrence-graph-in-poly-time}
Given a hypergraph $\mathcal{H} = (V,E)$, the co-occurrence graph of the dual hypergraph $\HH^d$ can be computed in time $\mathcal{O}(|V|^3|E|^3)$.
\end{corollary}

\begin{proof}
Immediate from  \Cref{dual-co-occurrence-graph-in-poly-time-improved} using the fact that the dimension and the maximum degree of $\mathcal{H}$ are bounded by $k \le |V|$ and $\Delta\le |E|$, respectively.
\end{proof}

\subsection{The {\sc Dual Conformality} problem}

We are interested in the complexity of testing conformality for the dual hypergraph of a given hypergraph $\HH$.
Formally, we introduce the following problem.

\begin{center}
\fbox{\parbox{0.85\linewidth}{\noindent
\textsc{Dual Conformality}\\[.8ex]
\begin{tabular*}{.93\textwidth}{rl}
{\em Input:} & A hypergraph $\HH$.\\
{\em Question:} & Is the dual hypergraph $\HH^d$ conformal?
\end{tabular*}
}}
\end{center}

\medskip

\Cref{prop:yes-instances-of-TC} has the following algorithmic consequence.

\begin{sloppypar}
\begin{proposition}\label{prop:maximal-cliques}
Given a hypergraph $\HH = (V,E)$ with dimension $k$ and maximum degree~$\gD$,
the {\sc Dual Conformality} problem is solvable
in time $\mathcal{O}(|V|^2(k|E|\Delta^2 + (|V|+|E|)\cdot|E(\HH^d)|))$.
\end{proposition}
\end{sloppypar}

\begin{proof}
First, we compute the co-occurrence graph $G=G(\HH^d)$ of $\HH^d$.
By \Cref{dual-co-occurrence-graph-in-poly-time-improved}, this can be done in time $\mathcal{O}(k|E|\Delta^2|V|^{2})$.
By \Cref{hyperedges are cliques in co-occurrence graph-corollary}, $\HH^d$ has only hyperedges that are cliques of $G$.
Now, the maximal cliques of $G$ can be generated with polynomial delay using the algorithm by Tsukiyama et al.~\cite{MR476582} on the complement of $G$.
More precisely, after a preprocessing step that takes $\mathcal{O}(|V|^2)$ time, the algorithm outputs all the maximal cliques of $G$ one by one, spending time $\mathcal{O}(|V|^3)$ between two consecutive output cliques.
We run the algorithm, and every time it outputs a maximal clique of $G$ check if it belongs to $\HH^d$ or not.
This is easy to check in time $\mathcal{O}(|V|^2|E|)$: it must be a transversal of $\HH$ and must be minimal.
If you get a NO at any time, then stop, and the answer is NO, otherwise, the answer is YES.
The total running time of this approach is $\mathcal{O}(k|E|\Delta^2|V|^{2}+|V|^2)+\mathcal{O}((|V|^3+|V|^2|E|)\cdot|E(\HH^d)|)$, which simplifies to $\mathcal{O}(k|E|\Delta^2|V|^{2} + |V|^2(|V|+|E|)\cdot|E(\HH^d)|)$.
\end{proof}

\begin{remark}
The approach of the proof of~\Cref{prop:maximal-cliques} actually shows the following.
Assume that there exists an algorithm for generating all maximal cliques of an $n$-vertex graph $G$ with preprocessing time $\mathcal{O}(T_1(n))$ and that spends time  $\mathcal{O}(T_2(n))$ between outputting any two consecutive maximal cliques.
Then, given a hypergraph $\HH = (V,E)$, the {\sc Dual Conformality} problem is solvable in time $\mathcal{O}(|V|^3|E|^3+T_1(|V|)+(|V|^2|E|+T_2(|V|))\cdot |E(\HH^d)|)$.
In particular, one could apply not only the algorithm by Tsukiyama et al.~but also any of the more recent faster algorithms, e.g., those in~\cite{MR4083179,MR2159537,MR3864713}.
\end{remark}

Of course, the size of $\HH^d$ could easily be exponential in the size of $\HH$, so this algorithm is exponential in the size of $\HH$, in the worst case.\footnote{Not on average, though.
On average, the size of the dual hypergraph of a Sperner hypergraph $\HH$ is polynomial in the size of $\HH$.
This follows from the proof of the main result in~\cite{MR0640855}.}
Accordingly, the question about computing $\HH^d$ from $\HH$ was typically addressed from the point of view of output-sensitive algorithms (see, e.g.,~\cite{MR584512,MR401486,MR1361157}).
The best currently known algorithm for computing $\HH^d$ for a general hypergraph $\HH$ has a running time which is linear in the output size and quasi-polynomial in the input size~\cite{MR1417667}.

\begin{observation}
The {\sc Dual Conformality} problem is in {\sf co-NP}.
\end{observation}

\begin{proof}
Suppose that for a given hypergraph $\HH$, its dual is not conformal.
Then there exists a maximal clique $C$ of the co-occurrence graph of $\HH^d$ that is not a minimal transversal of $\HH$.
It can be verified in polynomial time whether a set $C\subseteq V(\HH)$ satisfies all these properties.
By \Cref{dual-co-occurrence-graph-in-poly-time}, the co-occurrence graph $G(\HH^d)$ can be computed in polynomial time.
Having computed $G(\HH^d)$, we can check in polynomial time if every two distinct vertices in $C$ are adjacent in $G(\HH^d)$ and whether no vertex in $V(G(\HH^d))\setminus C$ is adjacent to all vertices in $C$.
Since the hypergraph $\HH$ is our input, we can also check in polynomial time if $C$ is not a minimal transversal of $\HH$.
\end{proof}

However, the complexity of {\sc Dual Conformality} remains open in general.

\medskip

\subsection{A polynomial case of {\sc Dual Conformality}}

We develop a polynomial-time algorithm for {\sc Dual Conformality} when restricted to the hypergraphs $\HH$ such that every maximal clique of the co-occurrence graph of $\HH^d$ is a transversal of $\HH$.
This algorithm is then used in \Cref{sec:small-upper-clique-transversal} to develop a polynomial-time algorithm for recognizing graphs in which all minimal clique transversals have size at most $k$, for every fixed~$k$.

\begin{center}
\fbox{\parbox{0.85\linewidth}{\noindent
\textsc{Restricted Dual Conformality}\\[.8ex]
\begin{tabular*}{.93\textwidth}{rl}
{\em Input:} & A hypergraph $\HH$ such that every maximal clique of $G(\HH^d)$\\& is a transversal of $\HH$.\\
{\em Question:} & Is the dual hypergraph $\HH^d$ conformal?
\end{tabular*}
}}
\end{center}

\begin{lemma}\label{characterization-of-no-instances-for-RTC}
Let $\HH$ be a hypergraph and let $G = G(\HH^d)$.
Suppose that every maximal clique of $G$ is a transversal of $\HH$.
Then $\HH$ is not dually conformal if and only if $G$ contains a vertex $v$ such that $N_G(v)$ is a transversal of $\HH$.
\end{lemma}

\begin{proof}
Assume first that there exists a vertex $v$ of $G$ such that $N_G(v)$ is a transversal of $\HH$.
Let $T$ be a minimal transversal of $\HH$ such that $T\subseteq N_G(v)$.
By \Cref{hyperedges are cliques in co-occurrence graph-corollary}, every minimal transversal of $\HH$ is a clique in $G$.
Thus, $T$ is a clique and since $T$ is contained in $N_G(v)$, the set $T \cup \{v\}$ is also a clique.
Let $C$ be a maximal clique in $G$ such that $T \cup \{v\} \subseteq C$.
Then $C$ is a maximal clique in $G$ that properly contains a minimal transversal of $\HH$ (namely $T$).
Therefore, $C$ is not a minimal transversal of $\HH$.
By \Cref{prop:yes-instances-of-TC}, $\HH$ is not dually conformal.

Assume now that $\HH$ is not dually conformal.
By \Cref{prop:yes-instances-of-TC}, $G$ has a maximal clique $C$ that is not a minimal transversal of $\HH$.
Since, by the assumption on $\HH$ every maximal clique of $G$ is a transversal of $\HH$, there exists a minimal transversal $T$ of $\HH$ properly contained in $C$.
Let $v$ be a vertex in $C\setminus T$.
Then, since $C$ is a clique, $T$ is a subset of $N_G(v)$.
This implies that $N_G(v)$ is a transversal of $\HH$.
\end{proof}

\begin{sloppypar}
\begin{proposition}\label{RTC-poly-time}
Given a hypergraph $\HH = (V,E)$ with dimension $k$ and maximum degree~$\gD$ such that every maximal clique of $G(\HH^d)$ is a transversal of $\HH$, the \textsc{Restricted Dual Conformality} problem is solvable in time $\mathcal{O}(k|E|\Delta^2|V|^{2})$.
\end{proposition}
\end{sloppypar}

\begin{proof}
Using \Cref{data-structure-for-hypergraphs}, we compute in time $\mathcal{O}(|V||E|)$ the edge-vertex incidence matrix of $\cH$ and the doubly linked representation of its incident pairs.
Next, we compute the co-occurrence graph $G=G(\HH^d)$ of $\HH^d$.
By \Cref{dual-co-occurrence-graph-in-poly-time-improved}, this can be done in time $\mathcal{O}(k|E|\Delta^2|V|^{2})$.
Then we iterate over all vertices $v$ of $G$ and verify in time $\mathcal{O}(k|E|)$ if the neighborhood of $v$ in $G$ is a transversal of $\HH$.
By~\Cref{characterization-of-no-instances-for-RTC}, if such a vertex exists, then $G$ is not dually conformal, and otherwise it is.
The total running time of this approach is $\mathcal{O}(k|E|\Delta^2|V|^{2}+k|V||E|) = \mathcal{O}(k|E|\Delta^2|V|^{2})$.
\end{proof}

\begin{remark}
The time complexity of the algorithm given by \Cref{RTC-poly-time} is dominated by the time needed to compute the co-occurrence graph of $\HH^d$.
The complexity of the remaining steps is only $\mathcal{O}(k|V||E|)$.
\end{remark}

\section{Graphs with small upper clique transversal number}\label{sec:small-upper-clique-transversal}

In this section we shift the focus from hypergraphs to graphs and apply the results from \Cref{sec:dually-conformal} to a problem about clique transversals in graphs.
Recall that a \emph{clique transversal} in a graph is a set of vertices intersecting all maximal cliques.
The problem of determining the minimum size of a clique transversal has received considerable attention in the literature (see, e.g., the works by Payan in 1979~\cite{MR539710}, by Andreae, Schughart, and Tuza in 1991~\cite{MR1099264}, and by Erd\H{o}s, Gallai, and Tuza in 1992~\cite{MR1189850}, as well as more recent works ~\cite{MR1201987,MR1375117,MR1413638,MR1423977,MR1737764,MR4213405,MR4264990,MR3875141,MR3350239,MR3325542,MR3131902,MR2203202}).
Recently, Milani\v{c} and Uno initiated in~\cite{MilanicUnoWG2023} the study of the ``upper'' variant of this parameter.
An \emph{upper clique transversal} of a graph $G$ is a minimal clique transversal of maximum size.
The \emph{upper clique transversal number} of a graph $G$ is denoted by $\tau_c^+(G)$ and defined as the maximum size of a minimal clique transversal in $G$.
In hypergraph terminology, the upper clique transversal number of a graph $G$ is the maximum size of a hyperedge of the dual of the clique hypergraph.
The corresponding decision problem is as follows.

\begin{center}
\fbox{\parbox{0.85\linewidth}{\noindent
\textsc{Upper Clique Transversal}\\[.8ex]
\begin{tabular*}{.93\textwidth}{rl}
{\em Input:} & A graph $G$ and an integer $k$.\\
{\em Question:} & Does $G$ contain a minimal clique transversal of size at least $k$,\\&
i.e., is $\tau_c^+(G)\ge k$?
\end{tabular*}
}}
\end{center}

Milani\v{c} and Uno showed in~\cite{MilanicUnoWG2023} that \textsc{Upper Clique Transversal} is \hbox{{\sf NP}-complete} in the classes of chordal graphs, chordal bipartite graphs, and line graphs of bipartite graphs, but solvable in linear time in the classes of split graphs and proper interval graphs.

We now show that for fixed $k$, the problem can be reduced in polynomial time to the \textsc{Restricted Dual Conformality} problem, and is thus polynomial-time solvable.
We consider the following family of problems parameterized by a positive integer $k$, where, unlike for the \textsc{Upper Clique Transversal} problem, $k$ is fixed and not part of the input.

\begin{center}
\fbox{\parbox{0.85\linewidth}{\noindent
\textsc{$k$-Upper Clique Transversal}\\[.8ex]
\begin{tabular*}{.93\textwidth}{rl}
{\em Input:} & A graph $G$.\\
{\em Question:} & Does $G$ contain a minimal clique transversal of size at least $k$,\\& i.e., is $\tau_c^+(G)\ge k$?
\end{tabular*}
}}
\end{center}

The problem is only interesting for $k\geq 2$, since every graph with at least one vertex is a yes-instance to the \textsc{$1$-Upper Clique Transversal} problem.

Let us first note that the variant of the \textsc{$k$-Upper Clique Transversal} problem in which the family of maximal cliques of the input graph $G$ is also part of the input admits a simple polynomial-time algorithm.
It suffices to verify if there exists a set $X\subseteq V(G)$ of size $k-1$ that is not a clique transversal of $G$ but is contained in some minimal clique transversal.
The former condition can be checked directly using the family of maximal cliques of $G$, and the latter condition can be checked in polynomial time since $k$ is fixed, by \Cref{subtransversal-running-time}.
An alternative solution would be to verify if there exists a set $X\subseteq V(G)$ of size $k$ that is contained in some minimal clique transversal.

\begin{sloppypar}
Solving the problem without knowing the family of maximal cliques (which could be exponential in the size of $G$) requires more work, but is still doable in polynomial time.
\end{sloppypar}

\begin{theorem}\label{poly-k-UCT}
For every integer $k\ge 2$, given a graph $G = (V,E)$, the \textsc{$k$-Upper Clique Transversal} problem is solvable in time $\mathcal{O}(|V|^{3k-3})$.
\end{theorem}

We prove \Cref{poly-k-UCT} in several steps. One key ingredient is a polynomial-time algorithm to test if a given constant-sized set of vertices in a graph is a clique transversal.\footnote{Note that the assumption on the bound on the size of the set is essential.
In fact, as shown by Zang~\cite{MR1344757}, it is {\sf co-NP}-complete to check, given a graph $G$ and an independent set $I$, whether $I$ is a clique transversal in $G$.}
By definition, a set $X$ of vertices in a graph $G$ is a clique transversal if and only if
$X$ intersects all maximal cliques. In particular, this means that for every clique $C$ in $G-X$ there exists a vertex $x\in X$ containing $C$ in its neighborhood.
As we show next, it is sufficient to require this condition for all
cliques $C$ in $G-X$ such that $|C|\le |X|$.

\begin{lemma}\label{thm:clique-transversals}
For every graph $G$ and every set $X\subseteq V(G)$, the following statements are equivalent.
\begin{enumerate}
  \item $X$ is a clique transversal in $G$.
  \item For every clique $C$ in $G-X$, there exists a vertex
$x\in X$ such that $C\subseteq N_G(x)$.
  \item For every clique $C$ in $G-X$ such that
$|C|\le |X|$, there exists a vertex
$x\in X$ such that $C\subseteq N_G(x)$.
\end{enumerate}
\end{lemma}

\begin{proof}
Suppose $X$ is a clique transversal in $G$ and let $C$ be a clique in $G-X$. Let $C'$ be a maximal clique in $G$ such that $C\subseteq C'$.
Then $C'$ contains a vertex $x\in X$. Since $C\cup \{x\}\subseteq C'$ and $C'$ is a clique, we must have $C\subseteq N_G(x)$.

Clearly, the second statement implies the third one.

We prove that the third statement implies the first one by contraposition. Suppose that $X$ is not a clique transversal in $G$. Then there
exists a maximal clique $C'$ in $G$ such that $C'\cap X = \emptyset$. Since $C'$ is a maximal clique disjoint from $X$, every vertex in $X$ has a non-neighbor in $C'$. Selecting one such non-neighbor for each vertex in $X$ results in a clique $C$ in $G-X$ such that
$|C|\le |X|$ and every vertex in $X$ has a non-neighbor in $C$.
Thus, there is no vertex $x\in X$ such that $C\subseteq N_G(x)$.
\end{proof}

\Cref{thm:clique-transversals} implies the following characterization of clique transversals of size one.
A \emph{universal vertex} in a graph $G$ is a vertex adjacent to all other vertices.

\begin{corollary}\label{cor:universal}
Given a graph $G= (V,E)$ and a vertex $v\in V$, the set $\{v\}$ is a clique transversal in $G$ if and only if $v$ is a universal vertex in $G$.
\end{corollary}

\begin{proof}
By \Cref{thm:clique-transversals}, the singleton $\{v\}$ is a clique transversal in $G$ if and only if for every clique $C$ in $G-v$, it holds that $C\subseteq N_G(v)$.
If this latter condition is satisfied, then $v$ is universal in $G$, since otherwise for any vertex $w$ in $G$ nonadjacent to $v$, the set $C = \{w\}$ would be a clique in $G-v$ that violates the condition $C\subseteq N_G(v)$.
And conversely, if $v$ is universal in $G$, then $N_G(v) = V(G)\setminus\{v\}$ and hence the condition $C\subseteq N_G(v)$ is satisfied trivially for any clique $C$ in $G-v$.
\end{proof}

As another consequence of \Cref{thm:clique-transversals}, we obtain that when the size of a set of vertices is bounded by a constant, testing whether the set is a clique transversal can be done in polynomial time.

\begin{proposition}\label{cor:constant}
For every fixed $k\ge 1$, there is an algorithm running in time $\mathcal{O}(|V|^{k})$ to check if, given a graph $G = (V,E)$ and a set $X\subseteq V(G)$ with $|X|\le k$, the set $X$ is a clique transversal of $G$.
\end{proposition}

\begin{proof}
If $k = 1$, then by \Cref{cor:universal} $X$ is a clique transversal of $G$ if and only if $X = \{v\}$ such that $v$ is a universal vertex in $G$.
This condition can be tested in time $\mathcal{O}(|V|)$.

Assuming $k \ge 2$, we first compute in time $\mathcal{O}(|V|^2)$ the adjacency matrix of $G$.
This will allow for testing adjacency of a pair of vertices in constant time.
By \Cref{thm:clique-transversals}, it suffices to verify if every clique $C$ in $G$ with size at most $|X|$ either contains a vertex of $X$ or is contained in the neighborhood of some vertex in $X$.
Since $|X|\le k$, all such cliques can be enumerated in time $\mathcal{O}(|V|^k)$.
For each such clique $C$, we can check in time $\mathcal{O}(|C||X|) = \mathcal{O}(1)$ if $C$ is disjoint from $X$.
If it is, then we iterate over all $\mathcal{O}(1)$ vertices $x\in X$ and for each such vertex $x$ check the condition $C\subseteq N_G(x)$ in time $\mathcal{O}(|C|) = \mathcal{O}(1)$.
If for some clique $C$ that is disjoint from $X$ no such vertex $x\in X$ exists, we conclude that $X$ is not a clique transversal in $G$, and otherwise it is.
The total running time is $\mathcal{O}(|V|^{k})$.
\end{proof}

Furthermore, note that for every fixed $k$, if a set $X\subseteq V(G)$ with $|X|\le k$ is a clique transversal of $G$, then we can check in polynomial time if $X$ is a minimal clique transversal simply by checking, for all $x\in X$, whether the set $X\setminus \{x\}$ is a clique transversal.
This can be done in time $\mathcal{O}(k|V|^{k-1}) = \mathcal{O}(|V|^{k-1})$ by \Cref{cor:constant}.

\begin{corollary}\label{cor:constant-minimal}
For every fixed $k$, there is an algorithm running in time $\mathcal{O}(|V|^{k})$ to check if, given a graph $G = (V,E)$ and a set $X\subseteq V(G)$ with $|X|\le k$, the set $X$ is a minimal clique transversal of $G$.
\end{corollary}

\begin{lemma}\label{lem:clique-transversals}
Let $k$ be a positive integer and $G$ be a graph.
Let $\HH$ be the hypergraph defined as follows: the vertex set of $\HH$ is $V(G)$, and the hyperedges of $\HH$ are precisely the minimal clique transversals $X$ of $G$ such that $|X|\le k$.
Then the following statements are equivalent.
\begin{enumerate}
\item\label{item1} $\tau_c^+(G)\le k$.
\item\label{item2} $\HH$ is the hypergraph of all minimal clique transversals of $G$.
\item\label{item3} $\HH$ is dually conformal and $G = G(\HH^d)$.
\end{enumerate}
\end{lemma}

\begin{sloppypar}
\begin{proof}
The equivalence between \cref{item1,item2} follows directly from the definition of $\tau_c^+(G)$.
We thus focus on establishing the equivalence between \cref{item2,item3}.

Assume that $\HH$ is the hypergraph of all minimal clique transversals of $G$, that is, $\HH$ is the dual hypergraph of the clique hypergraph of $G$.
By \Cref{duality-is-an-involution}, the dual hypergraph of $\HH$ is the clique hypergraph of $G$.
By \Cref{Sperner-conformal}, $\HH^d$ is conformal, that is, $\HH$ is dually conformal.
Furthermore, \Cref{a-technical-lemma} shows that $G = G(\HH^d)$.

Conversely, assume now that $\HH$ is dually conformal and $G = G(\HH^d)$.
Since $\HH^d$ is conformal, \Cref{Sperner-conformal} implies that $\HH^d$ is the clique hypergraph of $G(\HH^d) = G$.
Thus, by \Cref{duality-is-an-involution}, $\HH = (\HH^d)^d$ is the hypergraph of all minimal clique transversals of $G$.
\end{proof}
\end{sloppypar}

We now have everything ready to prove \Cref{poly-k-UCT}.

\begin{sloppypar}
\begin{proof}[Proof of \Cref{poly-k-UCT}.]
We first describe the algorithm and then justify its correctness and running time.
Let $G = (V,E)$ be the input graph.

The algorithm performs the following steps:
\begin{enumerate}
\item Compute the hypergraph $\HH$ defined as follows:
the vertex set of $\HH$ is $V$, and the hyperedges of $\HH$ are precisely the minimal clique transversals $X$ of $G$ such that $|X|< k$.
\item Compute the co-occurrence graph $G(\HH^d)$ of the dual hypergraph of $\HH$.
\item Check if $G\neq G(\HH^d)$.
\item If $G\neq G(\HH^d)$, then the algorithm determines that $\tau_c^+(G) \ge k$ (that is, $G$ is a yes-instance) and halts.
\item If $G= G(\HH^d)$, then apply \Cref{RTC-poly-time} on $\HH$ to test if $\HH^d$ is conformal.
\begin{itemize}
\item If $\HH^d$ is conformal, then the algorithm determines that $\tau_c^+(G) < k$ (that is, $G$ is a no-instance) and halts.
\item If $\HH^d$ is not conformal, then the algorithm determines that $\tau_c^+(G) \ge k$ (that is, $G$ is a yes-instance) and halts.
\end{itemize}
\end{enumerate}

To prove correctness, let us first justify that, in the case when $G= G(\HH^d)$, we can indeed apply \Cref{RTC-poly-time} on $\HH$ to test if $\HH^d$ is conformal.
By the definition of the hypergraph $\HH$, every maximal clique of $G$ intersects every hyperedge of $\HH$.
Thus, if $G= G(\HH^d)$, then every maximal clique of $G(\HH^d)$ is a transversal of $\HH$.
This means that $\HH$ is indeed a valid input to the {\sc Restricted Dual Conformality} problem, and hence  \Cref{RTC-poly-time} applies, as claimed.
Furthermore, by \Cref{lem:clique-transversals}, we have $\tau_c^+(G)< k$ if and only if $\HH$ is dually conformal and $G = G(\HH^d)$.
Equivalently, $\tau_c^+(G) \ge k$ if and only if one of the following conditions holds: either (i) $G \neq G(\HH^d)$ or (ii) $G = G(\HH^d)$ and $\HH^d$ is not conformal.
This implies that each of the three outputs of the algorithm is correct.

It remains to analyze the time complexity.
We compute the hypergraph $\HH$ in time $\mathcal{O}(|V|^{2k-1})$ by enumerating all the $\mathcal{O}(|V|^{k-1})$ subsets $X$ of $V$ with size less than $k$ and checking, for each such set $X$, if $X$ is a minimal clique transversal of $G$, in time $\mathcal{O}(|V|^{|X|}) =\mathcal{O}(|V|^{k-1})$ using~\Cref{cor:constant-minimal}.
Note that $\HH$ has $\mathcal{O}(|V|)$ vertices and $\mathcal{O}(|V|^{k-1})$ hyperedges.
Its dimension is at most $k-1$ and maximum degree $\Delta = \mathcal{O}(|V|^{k-2})$.
By \Cref{dual-co-occurrence-graph-in-poly-time-improved}, the co-occurrence graph of $\HH^d$ can be computed in time
$\mathcal{O}(k|E(\mathcal{H})|\Delta^2|V(\mathcal{H})|^{2}) = \mathcal{O}(k\cdot|V|^{k-1}\cdot |V|^{2(k-2)}\cdot |V|^{2})=\mathcal{O}(|V|^{3k-3})$.
The check whether $G$ and $G(\HH^d)$ are equal can be performed in time $\mathcal{O}(|V|+|E|)$ by comparing their adjacency lists.
Finally, testing conformality of $\HH^d$ in the case when the two graphs are the same can be done in time $\mathcal{O}(k|E(\mathcal{H})|\Delta^2|V(\mathcal{H})|^{2}) = \mathcal{O}(|V|^{3k-3})$ by \Cref{RTC-poly-time}.
As each of the remaining steps takes constant time, we conclude that the algorithm runs in time $\mathcal{O}(|V|^{3k-3})$.
\end{proof}
\end{sloppypar}

We close the section with a remark about the case $k = 2$.
Applying \Cref{poly-k-UCT} to this case shows that given a graph $G = (V,E)$, the \textsc{$2$-Upper Clique Transversal} problem is solvable in time $\mathcal{O}(|V|^{3})$.
However, the problem can be solved in linear time, as a consequence of the following characterization of graphs in which all minimal clique transversals have size one.

\begin{proposition}\label{lem:tau_c+=1}
Let $G$ be a graph. Then $\tau_c^+(G) = 1$ if and only if $G$ is complete.
\end{proposition}

\begin{proof}
If $G$ is complete, then the only minimal clique transversals are the sets consisting of a single vertex.
Thus, $\tau_c^+(G) = 1$ in this case.

Assume now that $G$ is not complete.
Let $S$ be the set of universal vertices of $G$.
Note that by \Cref{cor:universal}, $S$ is precisely the set of vertices $v$ such that $\{v\}$ is a clique transversal.
We claim that $V\setminus S$ is a clique transversal.
Suppose this is not the case.
Then $G$ admits a maximal clique $C$ contained entirely in $S$.
Since $S$ is a clique, we have $C = S$.
However, since every maximal clique contains $C$, it follows that $S$ is the only maximal clique in $G$ and hence $G$ is complete, a contradiction.
This shows that $V\setminus S$ is a clique transversal, as claimed.
Thus, $V\setminus S$ contains a minimal clique transversal, and any such clique transversal is of size at least $2$, since otherwise its only vertex would belong to $S$.
Consequently, $\tau_c^+(G) \ge 2$.
\end{proof}

\section{Dually conformal hypergraphs with bounded dimension}\label{sec:bounded-dimension}

In this section we study dually conformal hypergraphs of bounded dimension.
Recall that, given a hypergraph $\HH$, the \emph{dimension} of $\HH$ is the maximum cardinality of a hyperedge in $\HH$.

By \Cref{lem:clique-transversals-characterization}, a Sperner hypergraph is dually conformal if and only if there exists a graph $G$ such that $\HH$ is the hypergraph of all minimal clique transversals of~$G$.
In the case when dimension is bounded by a positive integer $k$, we obtain a similar characterization, which in addition takes into account the upper clique transversal number of graphs.

\begin{proposition}\label{lem:clique-transversals-characterization-dimension}
For every hypergraph $\HH$ and positive integer $k$, the following statements are equivalent.
\begin{enumerate}
\item $\HH$ is a dually conformal Sperner hypergraph with dimension at most $k$.
\item There exists a graph $G$ with $\tau_c^+(G)\le k$ such that $\HH$ is the hypergraph of all minimal clique transversals of~$G$.
\end{enumerate}
\end{proposition}

The proof of this proposition is very similar to the proof of \Cref{lem:clique-transversals-characterization}, so we omit it.

\medskip
For a positive integer $k$, we are interested in the complexity of the following problem.
\begin{center}
\fbox{\parbox{0.85\linewidth}{\noindent
\textsc{Dimension-$k$ Dual Conformality}\\[.8ex]
\begin{tabular*}{.93\textwidth}{rl}
{\em Input:} & A hypergraph $\HH$ with dimension at most $k$.\\
{\em Question:} & Is the dual hypergraph $\HH^d$ conformal?
\end{tabular*}
}}
\end{center}

\begin{sloppypar}
In this section we develop a polynomial-time algorithm for \textsc{Dimension-$k$ Dual Conformality} for any fixed positive integer $k$.
For the cases $k\in \{2,3\}$, we also develop more direct algorithms.
\end{sloppypar}

\subsection{The general case}

We start with a technical lemma.

\begin{sloppypar}
\begin{lemma}\label{condition (a) is easy for bounded dimension}
For every positive integer $k$, there exists an algorithm running in time \hbox{$\mathcal{O}(|E|\Delta^2|V|^{2}+|E||V|^{k})$} that takes as input a hypergraph $\HH = (V,E)$ with dimension at most~$k$ and maximum degree~$\gD$ and tests whether $G=G(\HH^d)$ contains a maximal clique $C$ that is not a transversal of $\HH$.
\end{lemma}
\end{sloppypar}

\begin{proof}
By \Cref{dual-co-occurrence-graph-in-poly-time-improved}, the graph $G$ can be computed in time $\mathcal{O}(k|E|\Delta^2|V|^{2})$, which is $\mathcal{O}(|E|\Delta^2|V|^{2})$ since the dimension is constant.
We show the existence of an algorithm with the stated running time that tests the negation of the stated condition, namely whether every maximal clique of $G$ is a transversal of $\HH$.
This condition is equivalent to the condition that every hyperedge of $\HH$ is a clique transversal in $G$.
Since each hyperedge $e$ of $\HH$ has size at most $k$, by \Cref{cor:constant} it can be tested in time $\mathcal{O}(|V|^{k})$ whether $e$ is a clique transversal in $G$.
Hence, the total running time of the described algorithm is
$\mathcal{O}(|E|\Delta^2|V|^{2}+|E||V|^{k})$.
\end{proof}

\begin{sloppypar}
\begin{theorem}\label{DkTC}
For every positive integer $k$, given a hypergraph $\HH = (V,E)$ with dimension at most $k$ and maximum degree~$\gD$, the \textsc{Dimension-$k$ Dual Conformality} problem is solvable in time \hbox{$\mathcal{O}(|E||V|^{2}\Delta^2+|E||V|^{k})$}.
\end{theorem}
\end{sloppypar}

\begin{sloppypar}
\begin{proof}
We make use of the characterization of dually conformal hypergraphs given by \Cref{no instances characterization 1}.
First we test condition (a) in time \hbox{$\mathcal{O}(|E||V|^{2}\Delta^2+|E||V|^{k})$} using \Cref{condition (a) is easy for bounded dimension}.
If condition (a) holds, we conclude that $\HH$ is not dually conformal.
If the condition does not hold, then every maximal clique of the graph $G = G(\HH^d)$ is a transversal of $\HH$, which means that $\HH$ is a valid input for the \textsc{Restricted Dual Conformality} problem.
In this case, we test dual conformality of $\HH$ in time $\mathcal{O}(k|E|\Delta^2|V|^{2})$ using \Cref{RTC-poly-time}.
Since $k$ is constant, the complexity simplifies to $\mathcal{O}(|E|\Delta^2|V|^{2})$.
\end{proof}
\end{sloppypar}

\subsection{The case of dimension three}

The case $k = 3$ of~\Cref{DkTC} is as follows.

\begin{sloppypar}
\begin{theorem}\label{D3TC}
Given a hypergraph $\HH = (V,E)$ with dimension at most $3$ and maximum degree $\Delta$, the \textsc{Dimension-$3$ Dual Conformality} problem is solvable in time \hbox{$\mathcal{O}(|E||V|^{2}\Delta^2+|E||V|^{3})$}.
\end{theorem}
\end{sloppypar}

We now develop an alternative approach for recognizing dually conformal hypergraphs within the family of hypergraphs of dimension at most three, based on a reduction to {\sc $2$-Satisfiability}.
The running time of this algorithm matches that of~\Cref{DkTC}.

Recall that \Cref{no instances characterization 1} gives two possible reasons why $\HH$ could fail to be dually conformal.
A similar characterization is as follows.

\begin{lemma}\label{no instances characterization 2}
Let $\HH$ be a hypergraph and let $G = G(\HH^d)$.
Then $\HH$ is not dually conformal if and only if one of the following two conditions holds.
\begin{enumerate}[(a)]
\item $G$ contains a maximal clique $C$ that is not a transversal of $\HH$, or
\item $G$ contains a clique $C$ and a vertex $v\in C$ such that for each hyperedge $e\in E(\HH)$ that contains $v$ we have $|C\cap e|\geq 2$.
\end{enumerate}
\end{lemma}

\begin{sloppypar}
\begin{proof}
By \Cref{no instances characterization 1}, the equivalence holds if $G$ contains a maximal clique $C$ that is not a transversal of $\HH$.

Suppose now that every maximal clique of $G$ is a transversal of $\HH$.
In this case, by \Cref{no instances characterization 1}, it suffices to show that $G$ contains a maximal clique $C$ that is a transversal of $\HH$ but not a minimal one if and only if $G$ contains a clique $C'$ and a vertex $v\in C'$ such that for all hyperedges $e\in E(\HH)$ that contain $v$ we have $|C'\cap e|\geq 2$.
Suppose first that $G$ contains a maximal clique $C$ that is a transversal of $\HH$ but not a minimal one.
Then there exists a vertex $v\in C$ such that $C\setminus\{v\}$ is a transversal of $\HH$.
In particular, this implies that for all hyperedges $e\in E(\HH)$ that contain $v$ we have $|C\cap e|\geq 2$.

For the converse direction, suppose that $G$ contains a clique $C'$ and a vertex $v\in C'$ such that for each hyperedge $e\in E(\HH)$ that contains $v$ we have $|C'\cap e|\geq 2$.
Let $C$ be a maximal clique in $G$ such that $C'\subseteq C$.
We claim that $C$ is a transversal of $\HH$ but not a minimal one.
The fact that $C$ is a transversal of $\HH$ follows from the assumption that  every maximal clique of $G$ is a transversal of $\HH$.
Furthermore, $C$ is not a minimal transversal since $C\setminus\{v\}$ is a transversal of $\HH$. To see this, consider an arbitrary hyperedge $e\in E(\HH)$.
\begin{itemize}
\item If $v\in e$, then $|C'\cap e|\geq 2$ and hence $|(C\setminus\{v\})\cap e|\ge |(C'\setminus\{v\})\cap e|\ge 1$.
\item If $v\not\in e$, then $(C\setminus\{v\})\cap e = C\cap e$, and $C\cap e\neq\emptyset$ since $C$ is a transversal of $\HH$.
\end{itemize}
Thus, in either case, $C\setminus \{v\}$ intersects $e$.
It follows that $C\setminus\{v\}$ is a transversal of $\HH$, as claimed.
\end{proof}
\end{sloppypar}

Let us now discuss the time complexity of verifying the two conditions from \Cref{no instances characterization 2}.
Recall that condition (a) can be tested in polynomial time for any bounded dimension using \Cref{condition (a) is easy for bounded dimension}.
Next we show that for hypergraphs with dimension at most three, condition (b) can be tested in polynomial time using a reduction to \textsc{$2$-Satisfiability}, a well-known problem solvable in linear time (see Aspvall, Plass, and Tarjan~\cite{MR526451}).

\begin{lemma}\label{condition (c) is easy for dimension at most 3}
There exists an algorithm running in time $\mathcal{O}(|E||V|^{2}\Delta^2+|V|^{3})$ that tests whether for a given hypergraph $\HH = (V,E)$ with dimension at most $3$ and maximum degree $\Delta$, the graph $G = G(\HH^d)$ contains a clique $C$ and a vertex $v\in C$ such that for each hyperedge $e\in E$ that contains $v$ it holds $|C\cap e|\geq 2$.
\end{lemma}

\begin{sloppypar}
\begin{proof}
By \Cref{dual-co-occurrence-graph-in-poly-time-improved} the co-occurrence graph $G=G(\HH^d)$ of $\HH^d$ can be constructed in time $\mathcal{O}(|E|\Delta^2|V|^{2})$.
We develop a polynomial-time algorithm to test, given a vertex $v$ of $G$, whether $G$ contains a clique $C$ such that $v\in C$ and for each hyperedge $e\in E$ that contains $v$ we have $|C\cap e|\geq 2$.
Let $e_1,\ldots, e_\ell$ be the hyperedges of $\HH$ that contain $v$.
We need to decide if there is a clique $K$ in $G$ such that
$K\subseteq N_G(v)$ and $K\cap e_i\neq\emptyset$ for all $i\in \{1,\ldots,\ell\}$.

For each $i\in \{1,\ldots,\ell\}$, we compute in time $\mathcal{O}(|V|)$ the intersection $e_i\cap N_G(v)$.
If \hbox{$e_i\cap N_G(v) = \emptyset$} for some $i\in  \{1,\ldots,\ell\}$, then the desired clique $K$ does not exist.
So let us assume that $e_i\cap N_G(v)\neq \emptyset$ for all $i\in  \{1,\ldots,\ell\}$.
In this case we determine the existence of a desired clique $K$ by solving the following instance of \textsc{$2$-Satisfiability}:
\begin{itemize}
\item For each vertex $u\in N_G(v)$ there is one variable $x_u$ (with the intended meaning that $x_u$ takes value true in a satisfying assignment if and only if $u\in K$).
\item For every two distinct non-adjacent vertices $u,w\in N_G(v)$, we introduce the clause $\neg x_u\vee \neg x_w$ (specifying that not both $u$ and $w$ can be selected in the clique $K$).

Furthermore, for every $i\in \{1,\ldots,\ell\}$, we introduce the clause $\bigvee_{u\in e_i\cap N_G(v)}x_u$ (specifying that at least one of the vertices in $e_i\cap N_G(v)$ should belong to $K$).
\end{itemize}

Note that for each $i\in  \{1,\ldots,\ell\}$, we have $v\in e_i$ and $|e_i|\le 3$ since $\HH$ has dimension at most~$3$.
Consequently, $|e_i\cap N_G(v)|\le |e_i\setminus\{v\}|\le 2$ and hence all the clauses have length one or two.
The instance of \textsc{$2$-Satisfiability} is constructed so that there is a clique $K$ in $G$ such that
$K\subseteq N_G(v)$ and $K\cap e_i\neq\emptyset$ for all $i\in \{1,\ldots,\ell\}$ if and only if the conjunction of all the clauses has a satisfying assignment.
There are $\mathcal{O}(\Delta)$ intersections $e_i\cap N_G(v)$, $i\in \{1,\ldots,\ell\}$, which can be computed in time $\mathcal{O}(\Delta|V|)$.
There are $\mathcal{O}(|V|)$ variables and $\mathcal{O}(|V|^2+\Delta)$ clauses, hence this is a polynomial-time reduction to the linear-time solvable \textsc{$2$-Satisfiability} problem.
We solve an instance of \textsc{$2$-Satisfiability} for each vertex $v$ of $G$, and hence the time complexity of this part of the algorithm is $\mathcal{O}(|V|(\Delta|V|+|V|+|V|^2+\Delta)) = \mathcal{O}(|V|^2(|V|+\Delta))$, resulting in the total running time of $\mathcal{O}(|E||V|^{2}\Delta^2+|V|^{3})$, as claimed.
\end{proof}
\end{sloppypar}

\Cref{no instances characterization 2,condition (a) is easy for bounded dimension,condition (c) is easy for dimension at most 3}
provide an alternative proof of \Cref{D3TC}.

\subsection{The two-uniform case}

In this section we analyze in more detail the case $k = 2$ of~\Cref{DkTC}, that is, the case of $2$-uniform hypergraphs.
Note that in this case we are dealing simply with graphs without isolated vertices; in particular, we shall also use the standard graph theory terminology and notation.

In the case $k = 2$, the characterization of dually conformal hypergraphs given by \Cref{no instances characterization 2} can be simplified as follows.

\begin{lemma}\label{lem:2-uniform-dually-conformal}
Let $\HH$ be a $2$-uniform hypergraph and let $G = G(\HH^d)$.
Then $\HH$ is not dually conformal if and only if one of the following two conditions holds.
\begin{enumerate}[(a)]
\item $G$ contains a maximal clique $C$ that is not a transversal of $\HH$, or
\item $G$ contains a vertex $v$ such that the closed neighborhood of $v$ in $\HH$ is a clique in $G$.
\end{enumerate}
\end{lemma}

\begin{proof}
By \Cref{no instances characterization 2}, it is sufficient to show that condition (b) from \Cref{lem:2-uniform-dually-conformal} is equivalent to condition (b) from \Cref{no instances characterization 2}.
For simplicity, let us refer to these two conditions as conditions (b$^*$) and (b), respectively.
Since $\HH$ is $2$-uniform, the inequality $|C\cap e|\geq 2$ in condition (b) is equivalent to the inclusion $e\subseteq C$.
Thus, condition (b) is equivalent to following condition:
$G$ contains a vertex $v$ and a clique $C$ such that $C$ contains $v$ as well as all hyperedges $e$ of $\HH$ that contain $v$.
In graph theoretic terms, this means that $G$ contains a vertex $v$ and a clique $C$ such that $C$ contains the closed neighborhood of $v$ in $\HH$.
If this condition is satisfied, then $N_\HH[v]$ is a clique in $G$, too, and condition (b$^*$) holds.
Conversely, if condition (b$^*$) holds and $v$ is a vertex in $G$ such that $N_\HH[v]$ is a clique in $G$, then we can take $C = N_\HH[v]$ and condition (b) is satisfied.
\end{proof}

Using \Cref{lem:2-uniform-dually-conformal} we now prove the announced result.

\begin{sloppypar}
\begin{theorem}\label{thm:2-uniform}
Given a $2$-uniform hypergraph $\HH = (V,E)$ with maximum degree $\Delta$, the \textsc{Dimension-$2$ Dual Conformality} problem is solvable in time \hbox{$\mathcal{O}(|E||V|^{2}\Delta^2)$}.
\end{theorem}
\end{sloppypar}

\begin{proof}
Let $\HH$ be the input $2$-uniform hypergraph and let $G = G(\HH^d)$ be the co-occurrence graph of $\HH^d$.
By \Cref{dual-co-occurrence-graph-in-poly-time-improved}, $G$ can be computed in time $\mathcal{O}(|E||V|^{2}\Delta^2)$.
By \Cref{lem:2-uniform-dually-conformal}, $\HH$ is not dually conformal if and only if one of the conditions $(a)$ and (b) from the lemma holds.
By \Cref{condition (a) is easy for bounded dimension}, condition (a) can be tested in time $\mathcal{O}(|E||V|^{2}\Delta^2$).
Since we know both graphs $G$ and $\HH$, condition (b) can also be tested in polynomial time: for each vertex $v$ of $G$, we compute the closed neighborhood of $v$ in $\HH$ and verify if it is a clique in $G$.
For a fixed vertex $v$ of $G$, this can be done in time $\mathcal{O}(\Delta^2)$, resulting in the total time complexity of $\mathcal{O}(|V|\Delta^2)$.
\end{proof}

\begin{remark}
The time complexity of the algorithm given by \Cref{thm:2-uniform} is dominated by the time needed to compute the co-occurrence graph of $\HH^d$.
The complexity of the remaining steps is only $\mathcal{O}(|E||V|^{2}+|V|\Delta^2)$.
\end{remark}

Recall that by \Cref{cor:universal}, a minimal clique transversal in a graph $G$ has size one if and only if it consists of a universal vertex.
Therefore, \Cref{lem:clique-transversals-characterization-dimension} and its proof imply the following.

\begin{corollary}\label{corollary-dually conformal-2-uniform}
For every $2$-uniform hypergraph $\HH$, the following statements are equivalent.
\begin{enumerate}
\item $\HH$ is dually conformal.
\item There exists a graph $G$ with $\tau_c^+(G)=2$ and without universal vertices such that $\HH$ is the hypergraph of all minimal clique transversals of~$G$.
\end{enumerate}
\end{corollary}

\section{Discussion}\label{sec:discussion}

We have initiated the study of dually conformal hypergraphs, that is, hypergraphs whose dual hypergraph is conformal.
As our main result, we developed a polynomial-time algorithm for recognizing dual conformality in hypergraphs of bounded dimension.

The main problem left open by our work is of course the problem of determining the complexity of {\sc Dual Conformality}.
In particular, the following questions are open.

\begin{question}
Is {\sc Dual Conformality} {\sf co-NP}-complete?
Is it in {\sf NP}?
Is it in {\sf P}?
\end{question}

\begin{sloppypar}
One could approach these questions by studying the {\sc Dual Conformality} problem in particular classes of hypergraphs, for example on hypergraphs derived from graphs (such as matching hypergraphs~\cite{MR2852510,MR818499}, and various clique~\cite{MR2755907,MR818499,MR1689294,MR539710,MR719998}, independent set~\cite{MR2755907,MR818499}, neighborhood~\cite{MR3314932,MR3281177}, separator~\cite{MR1095371,MR1936236,MR3963220}, and dominating set hypergraphs~\cite{MR3281177,MR3963220}, etc.).
If there exists a type of hypergraphs derived from graphs and a class of graphs $\mathcal{G}$ such that for each graph $G\in \mathcal{G}$, the corresponding hypergraph can be computed in polynomial time but testing dual conformality is {\sf co-NP}-complete, this would imply {\sf co-NP}-completeness of {\sc Dual Conformality}.
In particular, given that the conformality property of Sperner hypergraphs is closely related to clique hypergraphs of graphs (cf.~\Cref{Sperner-conformal}), it would be natural to investigate the complexity of {\sc Dual Conformality} when restricted to clique hypergraphs of graphs.
This leads to the following property of graphs.
A graph $G$ is \emph{clique dually conformal (CDC)} if its clique hypergraph is dually conformal.
\end{sloppypar}

\begin{question}
What is the complexity of recognizing CDC graphs?
\end{question}

\begin{sloppypar}
As explained above, the question is particularly interesting for graph classes with polynomially many maximal cliques.
As our preliminary investigations, we were able to develop polynomial-time algorithms for testing the CDC property in the classes of split graphs and triangle-free graphs.
To keep the length of this paper manageable, we shall present these results in a separate publication.
\end{sloppypar}

Recall that our results have implications for the upper clique transversal problem in graphs.
The variant of the problem in which $k$ is part of the input is known to be {\sf NP}-hard (see~\cite{MilanicUnoWG2023}).
In terms of the parameterized complexity of the problem (with $k$ as the parameter), \Cref{poly-k-UCT} shows that the problem is in {\sf XP}.
This motivates the following.

\begin{question}
Is the \textsc{$k$-Upper Clique Transversal} problem with $k$ as parameter {\sf W[1]}-hard?
\end{question}

We conclude with some structural questions.

\begin{question}\label{questionSpernerConformal}
Is there a real number $r\ge 1$ such that every  conformal hypergraph $\HH$ satisfies $(\dim(\HH)\cdot\dim(\HH^d))^r\ge |V(\mathcal{H})|$?
\end{question}

\begin{sloppypar}
Note that we may without loss of generality restrict our attention to Sperner conformal hypergraphs.
On the other hand, the conformality assumption in \Cref{questionSpernerConformal} is essential, as shown by the following construction by Vladimir Gurvich and Kazuhisa Makino (personal communication), generalizing a graph construction due to Costa, Haeusler, Laber, and Nogueira~\cite{MR2292971}.
Consider integers $d\ge 2$, $\ell\ge 1$, and $k>d$.
Define a $d$-uniform hypergraph $\mathcal{H}=(V,E)$ as follows.
Consider a set  $W$  of  $k$  vertices.
The hypergraph $\mathcal{H}$ contains, as hyperedges, all $d$-subsets of $W$  and
$\ell{k \choose d - 1}$ other edges, obtained as follows.
To every $(d-1)$-subset of $W$ let us assign a new vertex and add the obtained $d$-set to $E$.
Moreover, let us do this $\ell$ times for each  $(d-1)$-set.
Note that $\HH$ is not conformal.
Furthermore, the number of vertices is $|V| = k + \ell {k \choose d-1}$, while $\dim(\mathcal{H})= d$ and $\dim(\mathcal{H}^d) = k + \ell - (d-1)$.
In particular, taking an integer $q\ge 2$ and setting $d = q+1$, $k = 2q$, and $\ell = 1$, we obtain
$\dim(\mathcal{H})= \dim(\mathcal{H}^d) = q+1$, while $|V| = {2q \choose q}+2q$, which is exponential in $q$.
If  $d=2$, $k > d$ is arbitrary, and $\ell = k-1$, we obtain the same example as in~\cite{MR2292971}.
\end{sloppypar}

Since the general case of \Cref{questionSpernerConformal} is equivalent to the Sperner case, \Cref{Sperner-conformal} implies that the question can be posed equivalently in graph theoretic terms.
Recall that for a graph $G$, we denote by $\omega(G)$ the maximum size of a clique in $G$ and by $\tau_c^+(G)$ the upper clique transversal number of $G$.

\begin{question}\label{question:graphs}
Is there a real number $r\ge 1$ such that every graph $G$ satisfies
\[\left(\omega(G)\cdot \tau_c^+(G)\right)^r\ge |V(G)|\,?\]
\end{question}

A strongly related question for the class of CIS graphs (that is, graphs in which every maximal independent set is a clique transversal) was posed by Alc\'on, Gutierrez, and Milani\v{c} in~\cite{DBLP:journals/entcs/AlconGM19}.
Denoting by $\alpha(G)$ the maximum size of an independent set in a graph $G$, the question is as follows.

\begin{question}[Alc\'on, Gutierrez, and Milani\v{c}~\cite{DBLP:journals/entcs/AlconGM19}]\label{question:CIS-graphs}
Is there a real number $r\ge 1$ such that every CIS graph $G$ satisfies \[(\omega(G)\cdot \alpha(G))^r\ge |V(G)|\,?\]
\end{question}

\begin{sloppypar}
Note that a positive answer to \Cref{question:CIS-graphs} would imply a positive answer to \Cref{question:graphs} for the class of CIS graphs.
For general graphs, random graphs show that the analogue of \Cref{question:CIS-graphs} does not hold (see, e.g.,~\cite{MR1864966}).
On the other hand, the famous Erd\H{o}s-Hajnal conjecture (see, e.g., the survey by Chudnovsky~\cite{MR3150572}) states that the analogue of \Cref{question:CIS-graphs} holds when restricted to any class of graphs not containing a fixed graph $H$ as an induced subgraph (with the value of $r$ depending on $H$).
In contrast, every graph is an induced subgraph of a CIS graph (see~\cite{MR3838925}).
\end{sloppypar}

\subsection*{Acknowledgements}

\begin{sloppypar}
We are grateful to Kazuhisa Makino for helpful discussions related to \Cref{questionSpernerConformal}.
Part of the work for this paper was done in the framework of bilateral projects between Slovenia and the USA and between Slovenia and the Russian federation, partially financed by the Slovenian Research and Innovation Agency (BI-US/22--24--093, BI-US/22--24--149, BI-US/$20$--$21$--$018$, and BI-RU/$19$--$21$--$029$).
The work of the third author is supported in part by the Slovenian Research and Innovation Agency (I0-0035, research program P1-0285 and research projects J1-3001, J1-3002, J1-3003, J1-4008, and J1-4084), and by the research program CogniCom (0013103) at the University of Primorska. 
The research of the second author was included in the HSE University Basic Research Program.
The work of the fourth author is partially supported by JSPS KAKENHI Grant Number JP17K00017, 20H05964 and 21K11757, Japan.
\end{sloppypar}

\end{document}